\documentclass[draftclsnofoot,onecolumn,10pt]{IEEEtran}

\usepackage{amsmath,amssymb,enumerate,subfigure,multirow,hhline,color}
\usepackage{xcolor}
\usepackage{hyperref}
\usepackage{graphicx}
\usepackage{graphics}
\usepackage[sort]{cite}
\usepackage{amsthm}
\usepackage{algorithm}
\usepackage{algpseudocode}

\usepackage{tcolorbox}
\usepackage{epsfig,psfrag}
\usepackage{tabularx}
\usepackage{tabulary}

\usepackage[sort]{cite}

\usepackage{hyperref}
\hypersetup{breaklinks=true,colorlinks=true,linkcolor=blue,citecolor=blue}

\usepackage[noabbrev,capitalise,nameinlink]{cleveref}

\newtheorem{example}{Example}
\newtheorem{assumption}{Assumption}

\newtheorem{problem}{Problem}

\newtheorem{proposition}{Proposition}

\allowdisplaybreaks

\title{\LARGE \bf Multi-Objective LQG Design with Primal-Dual Method}

\author{Donghwan Lee and Do Wan Kim
\thanks{D. Lee is with the Department of Electrical Engineering,
KAIST, Daejeon, 34141, South Korea {\tt\small
donghwan@kaist.ac.kr}.}
\thanks{D. Kim is with the Department of Electrical Engineering, Hanbat National University, Daejeon 34158, South Korea {\tt\small
dowankim@hanbat.ac.kr}.}
}

\begin{document}

\maketitle

\begin{abstract}
The goal of this paper is to study a multi-objective linear quadratic Gaussian (LQG) control problem.
In particular, we consider an optimal control problem minimizing a quadratic cost over a finite time horizon for linear stochastic systems subject to control energy constraints. To solve the problem, we suggest an efficient bisection line search algorithm which is computationally efficient compared to other approaches such as the semidefinite programming. The main idea is to use the Lagrangian function and Karush–Kuhn–Tucker (KKT) optimality conditions to solve the constrained optimization problem. The Lagrange multiplier is searched using the bisection line search. Numerical examples are given to demonstrate the effectiveness of the proposed methods.
\end{abstract}

\section{Introduction}

Optimal control of dynamical systems has long been one of the fundamental problems in control community~\cite{bertsekas1996neuro,bertsekas2005dynamic}.
Among various optimal control scenarios, the linear quadratic Gaussian (LQG) control problem is our main concern. It covers various applications such as mobile robots, industrial quality control system, and flight control to name just a few. The LQG problem can be efficiently solved using the classical dynamic programming and Riccati equation. In many applications, however, there exist several possibly competing objectives that require behaviors that mediates among them. Sometimes, those multiple objectives can be formulated as constraints, for instance, bounds on different objectives, risk measures or costs. In the classical dynamic programming, multiple objectives can be encoded into the cost to be minimized. However, since they are blended into a single cost, designing a policy, which satisfies constraints or minimizes multiple objectives, is challenging. In this respect, an optimization-based multi-objective LQG design provide more potentials than the traditional dynamic programming based approaches by leveraging the existing constrained optimization algorithms and theories~\cite{Boyd2004}.

Emergence of convex optimization~\cite{Boyd2004} and semidefinite programming (SDP)~\cite{vandenberghe1996semidefinite} techniques in control analysis and design promoted new optimization formulations of control problems~\cite{Boyd1994,el2000advances}. They also have provided greater convenience and flexibility in control design with various objectives and constraints. For instance, SDP formulations of control problems with various constraints have been developed in several papers~\cite{gattami2010generalized,kothare1996robust,Boyd1994,primbs2009stochastic,costa1999constrained} to name just a few. In particular,~\cite{gattami2010generalized} proposed a new SDP formulation, where the finite-horizon LQG problem was converted into the optimal covariance matrix selection problem, and addressed energy constrained LQG. However, when the size of the problem is large, the computational complexity of such SDP-based algorithms is known to explode quickly, and it makes the problem numerically inefficient.

Motivated by the above discussions, the main goal of this paper is to study a numerically efficient algorithm for LQG problems with an energy constraint via optimization theory and Lagrangian duality. The problem has many applications such as the building control where limited resources are allowed for the control task. To find an optimal solution to the problem, we suggest a simple and efficient bisection line search algorithm whose computational complexity is in general lower than SDP-based methods. The main idea is to formulate a constrained optimization problem, and then use the Lagrangian function and Karush–Kuhn–Tucker (KKT) optimality conditions~\cite{luenberger1984linear} to solve the constrained optimization problem. The Lagrange multiplier is searched using the bisection line search. A numerical example of a building control problem is given to demonstrate the effectiveness of the proposed methods.

{\bf Notation}: The adopted notation is as follows: ${\mathbb N}$
and ${\mathbb N}_+$: sets of nonnegative and positive integers,
respectively; ${\mathbb R}$: set of real numbers; ${\mathbb R}_+$:
set of nonnegative real numbers; ${\mathbb R}_{++}$: set of
positive real numbers; ${\mathbb R}^n $: $n$-dimensional Euclidean
space; ${\mathbb R}^{n \times m}$: set of all $n \times m$ real
matrices; $A^T$: transpose of matrix $A$; $A^{-T}$: transpose of matrix $A^{-1}$; $A \succ 0$ ($A \prec
0$, $A\succeq 0$, and $A\preceq 0$, respectively): symmetric
positive definite (negative definite, positive semi-definite, and
negative semi-definite, respectively) matrix $A$; $I_n $: $n
\times n$ identity matrix; ${\mathbb S} ^n $: symmetric $n \times
n$ matrices; ${\mathbb S}_+^n $: cone of symmetric $n \times n$
positive semi-definite matrices; ${\mathbb S}_{++}^n$:
symmetric $n\times n$ positive definite matrices; ${\bf Tr}(A)$:
trace of matrix $A$.

\section{Finite-horizon LQG problem}

Consider the stochastic linear time-invariant (LTI) system
\begin{align}
&x(k+1)=Ax(k)+Bu(k)+w(k),\label{stochastic-LTI-system2}
\end{align}
where $k\in {\mathbb N}$, $x(k)\in {\mathbb R}^n$ is the state
vector, $u(k)\in {\mathbb R}^m$ is the input vector, $x(0)\sim
{\cal N}(z,V)$ and $w(k)\sim {\cal N}(0,W)$ are mutually independent Gaussian random vectors so that ${\mathbb E}[x(0)] = z$, ${\mathbb E}[w(k)] = 0$, ${\mathbb E}[(x(0)- z)(x(0) - z)^T ] = V$, and ${\mathbb E}[w(k)w(k)^T ] = W$. In this paper, we consider the following multi objective finite-horizon LQG problem:
\begin{problem}[Multi objective LQG problem]\label{finite-horizon-LQR-problem}
Solve
\begin{align*}
&\min_{F_0,\ldots,F_{N-1}\in
{\mathbb R}^{m \times n}}\,{\mathbb E}(x(k)^T Q_f x(k)) + \sum_{k=0}^{N-1}{{\mathbb E}\left(\begin{bmatrix}
   x(k)\\ u(k)\\
\end{bmatrix}^T \begin{bmatrix}
   Q_k & 0 \\ 0 & R_k  \\
\end{bmatrix} \begin{bmatrix}
   x(k)\\ u(k) \\
\end{bmatrix} \right)}\\
&{\rm subject}\,\,{\rm to}\\
&x(k+1)=Ax(k)+Bu(k)+w(k),\\
&u(k)=F_k x(k),\\
&{\mathbb E}(x(k)^T \tilde Q_f x(k)) + \sum\limits_{k = 0}^{N - 1} {{\mathbb E}\left( {\left[ {\begin{array}{*{20}c}
   {x(k)}  \\
   {u(k)}  \\
\end{array}} \right]^T \left[ {\begin{array}{*{20}c}
   {\tilde Q_k } & 0  \\
   0 & {\tilde R_k }  \\
\end{array}} \right]\left[ {\begin{array}{*{20}c}
   {x(k)}  \\
   {u(k)}  \\
\end{array}} \right]} \right)}  \le \gamma.
\end{align*}
\end{problem}

Note that the second objective is encoded into the inequality instead of the objective function. The problem may be useful in many optimal control applications, for example, the building control problem, where the goal is to reduce the indoor temperature tracking error as much as possible while using limited energy for the control input within a certain time horizon.
\begin{example}\label{ex1}
Consider a room's thermal dynamic model expressed
as~\eqref{stochastic-LTI-system2} with
\begin{align*}
&A = \begin{bmatrix}
   0.9500 & 0.0250 & 0.0250 & 0  \\
   0.0250 & 0.9750 & 0 & 0  \\
   0 & 0 & 1 & 0  \\
   0 & 0 & 0 & 1  \\
\end{bmatrix},\quad B = \begin{bmatrix}
   0.0250 \\
   0 \\
   0 \\
   0 \\
\end{bmatrix},
\end{align*}
where $x_1(k)$ is the indoor air temperature ($^{\circ}\mathrm{C}$), $x_2(k)$
is the wall temperature ($^{\circ}\mathrm{C}$), $x_3(k)$ is the outdoor air
temperature ($^{\circ}\mathrm{C}$), $x_4(k)$ is the reference temperature
($^{\circ}\mathrm{C}$). The outdoor air temperature and reference
temperature are kept constants ($30^{\circ}\mathrm{C}$ and $24^{\circ}\mathrm{C}$,
respectively) over time. To this end, the initial state should be deterministic and fixed, and the last element of the noise $w(k)$ should be zero over time. In this case, the initial state should be set to be
\[
x(0) = \left[ {\begin{array}{*{20}c}
   x_1(0) & x_2(0) & {30} & {24}  \\
\end{array}} \right]^T
\]
where $x_1(0)$ is the initial indoor air temperature, and $x_2(0)$ is the initial wall temperature. We want to enforce the indoor temperature to track the reference temperature $24^\circ C$ as close as possible while satisfying the total input energy constraint
\[
{\mathbb E}\left[ {\sum\limits_{k = 0}^{N - 1} {u(k)^2 } } \right] \le \gamma.
\]

The problem can be formulated as~\cref{finite-horizon-LQR-problem} with
\begin{align*}
Q_f  =& Q_k  = \left[ {\begin{array}{*{20}c}
   1 & 0 & 0 & { - 1}  \\
\end{array}} \right]^T \left[ {\begin{array}{*{20}c}
   1 & 0 & 0 & { - 1}  \\
\end{array}} \right],\\
R_k  =& 0,k \in \{ 0,1, \ldots ,N - 1\}
\end{align*}
and
\[
\tilde Q_f  = \tilde Q_k  = 0,R_k  = 1,k \in \{ 0,1, \ldots ,N - 1\}
\]
The cost function enforces the  indoor temperature to track the desired reference temperature.
\end{example}

Another important remark is that the control policy in~\cref{finite-horizon-LQR-problem} is restricted to a linear state-feedback law. It is well known that the optimal solution to LQG problem without the energy constraint is a linear state-feedback policy. However, more careful attention should be paid to the constrained case. If it admits a nonlinear optimal solution, then~\cref{finite-horizon-LQR-problem} may give only a suboptimal linear solution. Fortunately, the SDP design method in~\cite{gattami2010generalized} suggests that the optimal solution of the energy constrained problem is also linear so that no conservatism exists in~\cref{finite-horizon-LQR-problem}.

Throughout the paper, we assume that there exists a strictly feasible solution, i.e., there exists a solution such that the strict inequality constraint is satisfied. A collection of assumptions that will be used throughout the paper is summarized below.
\begin{assumption}\label{assumption:2}
The following assumptions are made:
\begin{enumerate}
\item $Q_f \succeq 0,\tilde Q_f \succeq 0,Q_k \succeq 0,\tilde Q_k \succeq 0,R_k + \lambda \tilde R_k  \succ 0$ for all $k \ge 0$ and $\lambda > 0$;

\item $V \succ 0$, $W \succ 0$.
\end{enumerate}
\end{assumption}

The assumptions $V \succ 0$ and $W \succ 0$ imply that all elements of $x(0)$ and $w(k)$ are stochastic. For the deterministic case, we need to set $V = 0$ and $W = 0$, and if some of the elements of $x(0)$ and $w(k)$ are partially deterministic, then we need to set $V\succeq 0$ and $W\succeq 0$. These cases will be briefly addressed later.

If we define the covariance of the augmented vector $[x(k)^T,u(k)^T]^T\in {\mathbb R}^{n \times m}$
\begin{align*}
&S_k={\mathbb E}\left( \begin{bmatrix} x(k)\\ u(k)\\ \end{bmatrix} \begin{bmatrix}
   x(k)\\
   u(k)\\
\end{bmatrix}^T \right),\quad k \in \{0,\ldots,N-1\},
\end{align*}
then,~\cref{finite-horizon-LQR-problem} can be equivalently
converted to the matrix equality constrained optimization problem or the covariance selection problem.
\begin{problem}[Covariance selection problem]\label{primal-LQR}
Solve
\begin{align*}
&J_p^*:=\min_{S_0,\ldots,S_{N-1}\in {\mathbb S}^{n+m},F_0,\ldots,F_{N-1}\in {\mathbb R}^{m \times n}}\,J_p(\{S_k\}_{k=0}^{N-1})\\
&{\rm subject}\,\,{\rm to}\\
&\Phi (F_k,S_{k-1})=S_k\quad k \in \{1,\ldots,N-1\},\\
&\begin{bmatrix}
   I_n\\
   F_0\\
\end{bmatrix} (V+ zz^T) \begin{bmatrix}
   I_n\\
   F_0\\
\end{bmatrix}^T = S_0,\\
&C(\{ S_k \} _{k = 0}^{N - 1} ) \le \gamma,
\end{align*}
where
\begin{align*}
J_p(\{S_k\}_{k=0}^{N-1}):=& {\bf Tr}\left(Q_f
\left(\begin{bmatrix}
   A^T\\B^T\\
\end{bmatrix}^T S_{N - 1} \begin{bmatrix}
   A^T\\B^T\\
\end{bmatrix}+W\right) \right) +\sum_{k=0}^{N-1}{{\bf Tr}\left( \begin{bmatrix}
 Q_k & 0\\ 0 & R_k\\
\end{bmatrix} S_k\right)}\\
C(\{ S_k \} _{k = 0}^{N - 1} ):=& {\bf{Tr}}\left( {\tilde Q_f \left( {\left[ {\begin{array}{*{20}c}
   {A^T }  \\
   {B^T }  \\
\end{array}} \right]^T S_{N - 1} \left[ {\begin{array}{*{20}c}
   {A^T }  \\
   {B^T }  \\
\end{array}} \right] + W} \right)} \right) + \sum\limits_{k = 0}^{N - 1} {{\bf{Tr}}\left( {\left[ {\begin{array}{*{20}c}
   {\tilde Q_k } & 0  \\
   0 & {\tilde R_k }  \\
\end{array}} \right]S_k } \right)}\\
\Phi (F, S):=&\begin{bmatrix}
   I_n \\F\\
\end{bmatrix}\left( \begin{bmatrix}
   A^T\\ B^T\\
\end{bmatrix}^T S \begin{bmatrix}
   A^T\\B^T\\
\end{bmatrix}+W\right)\begin{bmatrix}
   I_n\\F\\
\end{bmatrix}^T
\end{align*}
\end{problem}
In~\cref{primal-LQR}, the matrix equality constraints represent the covariance updates.
Since~\cref{finite-horizon-LQR-problem} is strictly feasible, we can prove that~\cref{primal-LQR} is also strictly feasible. Since this fact will be used later, we make a formal assumption for convenience.
\begin{assumption}[Strict feasibility]\label{assumption:strict-feasibility}
There exists at least one set of matrices $\{(S_k,F_k)\}_{k=0}^{N-1}$ such that all the equalities in~\cref{primal-LQR} are satisfied and all inequalities are strictly satisfied.
\end{assumption}
Based on the assumptions and definitions in this section, we will address the main results in the next section.

\section{Main Results}

In this section, we present main results of this paper. We first present the KKT condition~\cite{luenberger1984linear} for the optimization~\cref{primal-LQR}, and find potential optimal solution candidates satisfying the KKT condition. Then, we find the set of optimal solutions and their properties. Based on the analysis, a bisection algorithm is developed.

\subsection{Lagrangian solution}
For any $P_0,\ldots,P_{N-1}\in {\mathbb S}_+^{n+m}$ and
$\lambda \in {\mathbb R}_+$, define the Lagrangian function of~\cref{primal-LQR}
\begin{align*}
&L(\{(S_k,F_k,P_k)\}_{k=0}^{N-1},\lambda)\\
:=& J_p(\{S_k\}_{k=0}^{N-1})+\sum_{k=1}^{N-1} {{\bf Tr}((\Phi(F_k,S_{k-1})-S_k) P_k)}\\
&+{\bf Tr}\left( \left(\begin{bmatrix}
   I\\F_0\\
\end{bmatrix}(V +zz^T) \begin{bmatrix}
   I\\ F_0\\
\end{bmatrix}^T- S_0\right) P_0 \right) + \lambda (C(\{ S_k \} _{k = 0}^{N - 1} ) - \gamma ),
\end{align*}
where $P_0,\ldots,P_{N-1}\in {\mathbb S}_+^{n+m}$ are called the Lagrangian multipliers or dual variables.

Rearranging some terms, it can be rewritten as
\begin{align}
&L(\{( S_k,F_k,P_k\}_{k=0}^{N-1},\lambda)= J_d (\{P_k,F_k\}_{k=0}^{N-1})\nonumber\\
& + {\bf{Tr}}\left( {\left( {\left[ {\begin{array}{*{20}c}
   {A^T }  \\
   {B^T }  \\
\end{array}} \right](Q_f  + \lambda \tilde Q_f )\left[ {\begin{array}{*{20}c}
   {A^T }  \\
   {B^T }  \\
\end{array}} \right]^T  - P_{N - 1} }  { + \left[ {\begin{array}{*{20}c}
   {Q_{N - 1} } & 0  \\
   0 & {R_{N - 1} }  \\
\end{array}} \right] + \lambda \left[ {\begin{array}{*{20}c}
   {\tilde Q_{N - 1} } & 0  \\
   0 & {\tilde R_{N - 1} }  \\
\end{array}} \right]} \right)S_{N - 1} } \right)\nonumber\\
&+ \sum_{k = 1}^{N - 1} {{\bf{Tr}}((\Gamma _k (F_k ,P_k ,\lambda ) - P_{k - 1} )S_{k - 1} )}\label{eq1}
\end{align}
where
\begin{align*}
J_d (\{ P_k ,\,F_k \} _{k = 0}^{N - 1} ):=& {\bf{Tr}}\left( {\left[ {\begin{array}{*{20}c}
   I  \\
   {F_0 }  \\
\end{array}} \right](V +zz^T) \left[ {\begin{array}{*{20}c}
   I  \\
   {F_0 }  \\
\end{array}} \right]^T P_0 } \right) + \sum\limits_{k = 1}^N {{\bf{Tr}}\left( {\left[ {\begin{array}{*{20}c}
   I  \\
   {F_k }  \\
\end{array}} \right]W\left[ {\begin{array}{*{20}c}
   I  \\
   {F_k }  \\
\end{array}} \right]^T P_k } \right)}
\end{align*}
and
\begin{align*}
\Gamma _k (F,\,P,\lambda ): =& \left[ {\begin{array}{*{20}c}
   {A^T }  \\
   {B^T }  \\
\end{array}} \right]\left[ {\begin{array}{*{20}c}
   I  \\
   F  \\
\end{array}} \right]^T P\left[ {\begin{array}{*{20}c}
   I  \\
   F  \\
\end{array}} \right]\left[ {\begin{array}{*{20}c}
   {A^T }  \\
   {B^T }  \\
\end{array}} \right]^T \\
&+ \left[ {\begin{array}{*{20}c}
   {Q_k } & 0  \\
   0 & {R_k }  \\
\end{array}} \right] + \lambda \left[ {\begin{array}{*{20}c}
   {\tilde Q_k } & 0  \\
   0 & {\tilde R_k }  \\
\end{array}} \right]
\end{align*}

Based on the Lagrangian function, the KKT condition can be summarized as
\begin{enumerate}
\item Primal feasibility condition:
\begin{align}
&\begin{bmatrix}
   I  \\
   {F_0 }  \\
\end{bmatrix} (V + zz^T) \begin{bmatrix}
   I  \\
   {F_0 }  \\
\end{bmatrix}^T = S_0,\nonumber\\
&\Phi (F_k ,S_{k - 1} )= S_k,\quad k \in \{ 1,2, \ldots ,N - 1\}\label{eq:10}\\
&\tilde J_p (\{ S_k \} _{k = 0}^{N - 1} ) \le \gamma.\label{eq:11}
\end{align}

\item Complementary slackness condition:
\begin{align}
\lambda (\tilde C(\{ S_k \} _{k = 0}^{N - 1} ) - \gamma ) = 0.\label{eq:12}
\end{align}

\item Dual feasibility condition:
\begin{align}
\lambda  \ge 0.\label{eq:13}
\end{align}

\item Stationary condition $\nabla _{S_k ,F_k } {\cal L}(\{ (S_k ,\,F_k ,\,P_k )\} _{k = 0}^{N - 1} ,\lambda ) = 0$:
\begin{align}
&P_N  = \begin{bmatrix}
   {Q_f + \lambda \tilde Q_f} & 0  \\
   0 & 0  \\
\end{bmatrix},\quad \Gamma_N (0,P_N,\lambda) = P_{N-1}\label{eq:6}\\
&\Gamma_k (F_k ,P_k,\lambda) =
P_{k - 1} ,\quad k \in \{ 1,2, \ldots N - 1\}\label{eq:7}\\
&(V+zz^T) (P_{0,12}  + F_0^T P_{0,22}) + (P_{0,12}^T + P_{0,\,22} F_0 )(V+zz^T)  = 0\label{eq:8}\\
&M_k (P_{k + 1,12}  + F_{k + 1}^T P_{k + 1,22} ) + (P_{k + 1,12}^T  + P_{k + 1,22} F_{k + 1} )M_k  = 0\label{eq:9}\\
&k \in \{ 1,2, \ldots ,N - 1\}\nonumber
\end{align}
where $M_k  = \begin{bmatrix}
   A & B  \\
\end{bmatrix} S_k \begin{bmatrix}
   A & B  \\
\end{bmatrix}^T + W$.

\end{enumerate}

Using the KKT condition, we establish a modified Riccati equation for solving the multi-objective problem in the following.
\begin{proposition}\label{theorem1}
Suppose that $\lambda \ge 0 $ is fixed and arbitrary. Consider the Riccati equation
\begin{align}
&A^T X_{k+1}^\lambda A-A^T X_{k+1}^\lambda B(R_k + \lambda \tilde R_k +B^T X_{k+1}^\lambda B)^{-1} B^T X_{k+1}^\lambda A + Q_k + \lambda \tilde Q_k=X_k^\lambda\label{Riccati}
\end{align}
for all $k \in \{0,\ldots,N-1\} $ with $X_N^\lambda =Q_f + \lambda \tilde Q_f$,
and define $\{( S_k^\lambda,F_k^\lambda,P_k^\lambda)\}_{k=0}^{N-1}$ with
\begin{align}
F_k^\lambda=&-(R_k + \lambda \tilde R_k+B^T X_{k+1}^\lambda B)^{-1} B^T X_{k+1}^\lambda A,\nonumber\\
S_k^\lambda=& \Phi (F_k^\lambda,S_{k-1}^\lambda),\quad S_0=\begin{bmatrix}
I\\ F_0^\lambda\\
\end{bmatrix} (V+zz^T) \begin{bmatrix}
 I\\ F_0^\lambda\\
\end{bmatrix}^T,\nonumber\\
P_k^\lambda =& \begin{bmatrix}
   Q_k + \lambda \tilde Q_k + A^T X_{k+1}^\lambda A & A^T X_{k+1}^\lambda B\\
   B^T X_{k+1}^\lambda A & R_k  + \lambda \tilde R_k + B^T X_{k+1}^\lambda B\\
\end{bmatrix},\label{primal-optimal-point}
\end{align}
where the superscript $\lambda$ is included to designate the dependence on $\lambda$. Then, $\{(S_k^\lambda,F_k^\lambda)\}_{k=0}^{N-1}$ is a primal feasible point of~\cref{primal-LQR} uniquely satisfying the primal feasibility condition~\eqref{eq:10} and $\{(P_k^\lambda,F_k^\lambda) \}_{k=0}^{N-1}$ uniquely satisfies the stationary condition~\eqref{eq:6}-\eqref{eq:9}.
\end{proposition}
\begin{proof}
Using~\cref{assumption:2}, $V \succ 0$ implies that $V + zz^T$ is nonsingular, and consequently,~\eqref{eq:8} implies $P_{0,12}^T  + P_{0,\,22} F_0 =0$. Similarly, $W \succ 0$ with~\eqref{eq:9} implies $P_{k,12}^T  + P_{k,\,22} F_k  = 0$ for all $k \in \{0,1,\ldots, N-1\}$. On the other hand,~\eqref{eq:6} and~\eqref{eq:7} with the assumption that $R_k + \lambda \tilde R_k \succ 0 $ for any $\lambda >0$ in~\cref{assumption:2} ensure $P_{k,\,22}$ is nonsingular for all $k \in \{0,1,\ldots, N\}$. Therefore, the feedback gains are uniquely determined by $F_k  =  - P_{k,\,22}^{ - 1} P_{k,12}^T$ for all $k \in \{0,1,\ldots, N-1\}$. Plugging this expression into~\eqref{eq:6} and~\eqref{eq:7} leads to the construction in~\eqref{primal-optimal-point} with the Riccati equation~\eqref{Riccati}. Note that under~\cref{assumption:2} and fixed $\lambda > 0$, the KKT point is uniquely determined.
\end{proof}

If the inequality constraint is removed and $\lambda = 0$, then the Riccati equation in~\eqref{Riccati} is reduced to the standard Riccati equation. In this case, it is clear that the solution which satisfies the KKT condition is unique. Therefore, the solution obtained form the Riccati equation is the unique optimal solution, which is a well-known fact.
\begin{proposition} Consider~\cref{primal-LQR} without the inequality constraint. Then, $\{(S_k^0,F_k^0)\}_{k=0}^{N-1}$ is the unique optimal solution of~\cref{primal-LQR}.
\end{proposition}
\begin{proof}
It is clear that the tuples $\{S_k^0,F_k^0,P_k^0\}_{k=0}^{N-1}$ uniquely satisfy the KKT condition. Since the KKT condition is a necessary condition for optimality, it is a unique optimal solution of~\cref{primal-LQR} without the inequality constraint.
\end{proof}

\cref{theorem1} tells us that the Riccati equation can be induced from the Lagrangian function and KKT condition in optimization theory instead of the classical argument from the value function and HJB equation. Moreover, we can see that the solution of the multi-objective LQG defined in~\cref{finite-horizon-LQR-problem} is nothing but the solution of a standard LQG problem with modified weight $Q_f  + \lambda \tilde Q_f ,Q_k  + \lambda \tilde Q_k ,R_k  + \lambda \tilde R_k$ and an appropriately chosen $\lambda >0$.
Let us now focus on how to determine the Lagrange multiplier $\lambda$ satisfying the KKT condition. We need to consider the following three scenarios:
\begin{enumerate}
\item If the strict inequality $C(\{ S_k^0 \} _{k = 0}^{N - 1} ) < \gamma$ is already satisfied with $\{ F_k^0 \} _{k = 0}^{N - 1}$ obtained using the standard Riccati equation, then $\lambda = 0$ solves the complementary slackness condition. We do not need to do anything in this case.

\item Moreover, if the equality $C(\{ S_k^0 \} _{k = 0}^{N - 1} ) = \gamma$ is satisfied with $\{ F_k^0 \} _{k = 0}^{N - 1}$ obtained using the standard Riccati equation, then any $\lambda\ge 0$ solves the complementary slackness condition. However, when $\lambda > 0$, the corresponding $\{ S_k^\lambda,F_k^\lambda ,P_k^\lambda \} _{k = 0}^{N - 1}$ may be different from $\{ S_k^0,F_k^0,P_k^0\} _{k = 0}^{N - 1}$. Therefore, to use the variables obtained in~\cref{theorem1} as a solution to the KKT condition, we need to set $\lambda = 0$.

\item Lastly, assume that $C (\{ S_k^0 \} _{k = 0}^{N - 1} ) > \gamma$ holds with $\{ F_k^0 \} _{k = 0}^{N - 1}$ obtained using the standard Riccati equation. Then, some $\lambda > 0$ solves the complementary slackness condition if $C (\{ S_k^\lambda  \} _{k = 0}^{N - 1} ) = \gamma$. Suppose that $\lambda^*>0$ is such a number. Then, the corresponding tuple $(\lambda ^* ,S_k^{\lambda ^* } ,F_k^{\lambda ^* } ,P_k^{\lambda ^* } )$ satisfies the KKT condition.
\end{enumerate}

For simplicity of the presentation, we only focus on the last case because the other cases are trivial, and we formalize it in the following assumption.
\begin{assumption}[Nontrivial scenario]\label{assumption:1}
Throughout the paper, we assume that  $C(\{ S_k^0 \} _{k = 0}^{N - 1} ) > \gamma$ holds with $\{ F_k^0 \} _{k = 0}^{N - 1}$ obtained using the standard Riccati equation.
\end{assumption}

To proceed further, we need to establish some properties of the function $f:{\mathbb R}_ +   \to {\mathbb R}$ defined as
\[
f(\lambda ) := C(\{ S_k^\lambda  \} _{k = 0}^{N - 1} ) - \gamma, \quad \lambda \ge 0,
\]
which evaluates the error in the inequality constraint. In the following, we study various properties of $f$ which play important roles throughout this paper.
\begin{proposition}[Properties of $f$]\label{proposition:1}
Define the function $f:\mathbb R_+ \to \mathbb R$ as
\[
f(\lambda ) := C(\{ S_k^\lambda  \} _{k = 0}^{N - 1} ) - \gamma
\]
Then, the following statements hold:
\begin{enumerate}
\item $f$ is continuous over $\mathbb R_+$;

\item $f(\lambda ) \le f(\lambda  + \varepsilon )$  holds for any $\varepsilon >0$;

\item If $f(\lambda ) = f(\lambda  + \varepsilon )$ holds for some $\varepsilon >0$, then $J_p (\{ S_k^\lambda  \} _{k = 0}^{N - 1} ) = J_p (\{ S_k^{\lambda + \varepsilon}  \} _{k = 0}^{N - 1} )$;

\item $f(0)>0$;

\item There exists a $\lambda >0$ such that $f(\lambda)<0$;

\item Define the set-valued mapping $T:(V,W) \mapsto \{ \lambda  > 0:f(\lambda ) = 0\}$. Then, $T(V,W)$ is a closed line segment.

\end{enumerate}
\end{proposition}
\begin{proof}
\begin{enumerate}
\item From the definition, $P_k^\lambda$ is linear in $\lambda$, $F_k^\lambda$ is rational, whose entries are finite for a finite $\lambda \in \mathbb R_+$ because the inverse matrix $(R_k + \lambda \tilde R_k+B^T X_{k+1}^\lambda B)^{-1}$ in $F_k^\lambda=-(R_k + \lambda \tilde R_k+B^T X_{k+1}^\lambda B)^{-1} B^T X_{k+1}^\lambda A$ is finite for all $\lambda \in \mathbb R_+$. Therefore, from the definition, $S_k^\lambda$ is also rational and finite over $\lambda \in \mathbb R_+$, which implies that $S_k^\lambda$ is continuous in $\lambda  \in \mathbb R_+$. This completes the proof.

\item We only need to prove the inequality for $C(\{ S_k^\lambda  \} _{k = 0}^{N - 1} )$. By contradiction, suppose that $C(\{ S_k^{\lambda  + \varepsilon } \} _{k = 0}^{N - 1} ) > C(\{ S_k^\lambda  \} _{k = 0}^{N - 1} )$ holds. For a fixed $\lambda$, we see from the KKT condition that the problem is nothing but the optimization
\begin{align}
&J_p^*:=\min_{S_0,\ldots,S_{N-1}\in {\mathbb S}^{n+m},F_0,\ldots,F_{N-1}\in {\mathbb R}^{m \times n}}\,J_p(\{S_k\}_{k=0}^{N-1})\nonumber\\
& + \lambda[C(\{ S_k \} _{k = 0}^{N - 1} ) - \gamma ]\label{eq:4}\\
&{\rm subject}\,\,{\rm to}\nonumber\\
&\Phi (F_k,S_{k-1})=S_k\quad k \in \{1,\ldots,N-1\},\nonumber\\
&\begin{bmatrix}
   I_n\\
   F_0\\
\end{bmatrix} W_f \begin{bmatrix}
   I_n\\
   F_0\\
\end{bmatrix}^T = S_0\nonumber
\end{align}
with an augmented objective. Since $\{ S_k^\lambda  \} _{k = 0}^{N - 1}$ is the optimal solution corresponding to $\lambda$, it follows that
\begin{align}
&J_p (\{ S_k^\lambda  \} _{k = 0}^{N - 1} ) + \lambda [C (\{ S_k^\lambda  \} _{k = 0}^{N - 1} ) - \gamma ]\le J_p (\{ S_k^{\lambda  + \varepsilon } \} _{k = 0}^{N - 1} ) + \lambda [C (\{ S_k^{\lambda  + \varepsilon } \} _{k = 0}^{N - 1} ) - \gamma ]\label{eq2}
\end{align}
where $\{ S_k^{\lambda + \varepsilon}  \} _{k = 0}^{N - 1}$ is the optimal solution corresponding to $\varepsilon \leftarrow \lambda + \varepsilon$. On the other hand, we have
\begin{align}
&J_p (\{ S_k^{\lambda  + \varepsilon } \} _{k = 0}^{N - 1} ) + (\lambda  + \varepsilon )[C(\{ S_k^{\lambda  + \varepsilon } \} _{k = 0}^{N - 1} ) - \gamma ] \le J_p (\{ S_k^\lambda  \} _{k = 0}^{N - 1} ) + (\lambda  + \varepsilon )[C(\{ S_k^\lambda  \} _{k = 0}^{N - 1} ) - \gamma ]\label{eq:5}
\end{align}
which leads to
\begin{align}
&J_p (\{ S_k^{\lambda  + \varepsilon } \} _{k = 0}^{N - 1} ) + \lambda [C(\{ S_k^{\lambda  + \varepsilon } \} _{k = 0}^{N - 1} ) - \gamma ]\nonumber\\
\le& J_p (\{ S_k^\lambda  \} _{k = 0}^{N - 1} ) + \lambda [C(\{ S_k^\lambda  \} _{k = 0}^{N - 1} ) - \gamma ] + \varepsilon [C(\{ S_k^\lambda  \} _{k = 0}^{N - 1} ) - C(\{ S_k^{\lambda  + \varepsilon } \} _{k = 0}^{N - 1} )]\label{eq3}
\end{align}

Combining~\eqref{eq3} with~\eqref{eq2} yields
\[
0 \le \varepsilon [C(\{ S_k^\lambda  \} _{k = 0}^{N - 1} ) - C(\{ S_k^{\lambda  + \varepsilon } \} _{k = 0}^{N - 1} )]
\]
which contradicts with our hypothesis. This completes the proof.

\item Assume $f(\lambda ) = f(\lambda  + \varepsilon )$ holds for some $\varepsilon >0$. Then,~\eqref{eq2} leads to $J_p (\{ S_k^\lambda  \} _{k = 0}^{N - 1} ) \le J_p (\{ S_k^{\lambda  + \varepsilon } \} _{k = 0}^{N - 1} )$, while~\eqref{eq:5} yields $J_p (\{ S_k^{\lambda  + \varepsilon } \} _{k = 0}^{N - 1} ) \le J_p (\{ S_k^\lambda  \} _{k = 0}^{N - 1} )$. Combining the two inequalities leads to the desired conclusion.

\item The fourth statement is true due to~\cref{assumption:1}.

\item Note that the objective in~\eqref{eq:4} can be replaced with $\lambda ^{ - 1} J_p (\{ S_k^\lambda \} _{k = 0}^{N - 1} ) + C(\{ S_k^\lambda \} _{k = 0}^{N - 1} ) - \gamma$ without changing the optimal solutions.
  As $\lambda \to \infty$, the objective converges to $C(\{ S_k^\lambda \} _{k = 0}^{N - 1} ) - \gamma = f(\lambda)$, which implies $C(\{ S_k^\lambda \} _{k = 0}^{N - 1} ) - \gamma = f(\lambda) < 0$ as $\lambda \to \infty$ due to the strict feasibility assumption in~\cref{assumption:strict-feasibility}. Since $f$ is continuous over $\mathbb R_+$ from the first statement, there should exists $\lambda >0 $ such that $f(\lambda) <0$.

\item Define $a = \arg\sup \{ \lambda  > 0:f(\lambda ) = 0\}$ and $b = \arg\inf \{ \lambda  > 0:f(\lambda ) = 0\}$. From the continuity of $f$, the supremum and infimum are attained; otherwise, $f$ should be discontinuous. Therefore, we can define $a = \max \{ \lambda  > 0:f(\lambda ) = 0\}$ and $b = \min \{ \lambda > 0:f(\lambda ) = 0\}$. From the second statement, we see that $f(\lambda) = 0$ for all $\lambda \in [b,a]$. It completes the proof.

\end{enumerate}
\end{proof}
\cref{proposition:1} suggests that $f$ is monotonically decreasing over the nonnegative real numbers. Moreover, we can choose a $\bar\lambda >0 $ such that $f(\lambda ) < 0,\forall \lambda  \in [\bar \lambda ,\infty )$. Let $\tilde \lambda \in [\bar \lambda ,\infty )$. Then, $f(\lambda)$ is a monotonically decreasing over $\lambda \in [0,\tilde \lambda]$, which connects $f(\tilde \lambda)<0$ and $f(0)>0$. The graph of $f$ is illustrated in the following example.
\begin{example}
Let us consider~\cref{ex1} again. With $\gamma = 3$ and $N=10$, the function value $f(\lambda)$ is plotted in~\cref{fig:1}, demonstrating the monotonically non-increasing property and the zero crossing property in~\cref{proposition:1}.
\begin{figure}[t]
\centering\includegraphics[width=8cm,height=6cm]{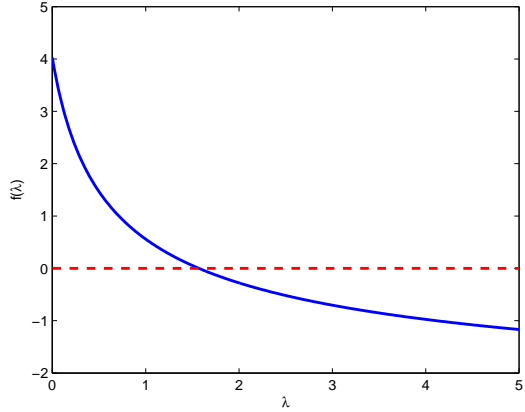}
\caption{$f(\lambda)$ for different $\lambda \in [0,5]$}\label{fig:1}
\end{figure}
\end{example}

Based on~\cref{proposition:1}, we can characterize the set of the KKT points. In particular, it turns out that the KKT point corresponding to $\lambda$ should satisfies the constraint $C(\{ S_k^\lambda \} _{k = 0}^{N - 1} ) = \gamma$·
\begin{proposition}\label{thm:1}
The set of variables satisfying the KKT condition is all tuples $(\lambda ,S_k^\lambda  ,F_k^\lambda  ,P_k^\lambda  )_{k = 0}^{N - 1}$ such that $\lambda>0$ and $C(\{ S_k^\lambda \} _{k = 0}^{N - 1} ) = \gamma$.
\end{proposition}
\begin{proof}
If $C(\{ S_k^\lambda \} _{k = 0}^{N - 1} ) > \gamma$, the KKT condition is obviously not satisfied because the complementary slackness condition~\eqref{eq:12} is not satisfied with $\lambda>0$. If $C(\{ S_k^\lambda \} _{k = 0}^{N - 1} ) < \gamma$, then the complementary slackness condition is not satisfied with $\lambda >0$. Only the case that KKT is satisfied with $\lambda >0$ is the case that $C(\{ S_k^\lambda \} _{k = 0}^{N - 1} ) = \gamma$ holds. This completes the proof.
\end{proof}

\cref{thm:1} gives us a clue on how to decide the KKT point. However, since the KKT condition is a necessary condition for the optimality, there is no guarantee that a KKT point found is actually an optimal solution. Fortunately, we can prove that all the KKT points characterized in~\cref{thm:1} constitute the optimal solutions.
\begin{proposition}\label{thm:2}
Consider any tuples $(\lambda ,S_k^\lambda  ,F_k^\lambda  ,P_k^\lambda  )_{k = 0}^{N - 1}$ such that $\lambda>0$ and $C(\{ S_k^\lambda \} _{k = 0}^{N - 1} ) = \gamma$. The set of such tuples is formally defined as
\[
\Lambda : = \{ (\lambda ,S_k^\lambda  ,F_k^\lambda  ,P_k^\lambda  )_{k = 0}^{N - 1} :\lambda  > 0, C(\{ S_k^\lambda  \} _{k = 0}^{N - 1} ) = \gamma \}
\]
Then, the corresponding $\{(S_k^\lambda,F_k^\lambda)\}_{k=0}^{N-1}$ with $(\lambda ,S_k^\lambda  ,F_k^\lambda  ,P_k^\lambda  )_{k = 0}^{N - 1} \in \Lambda$ is an optimal solution of the constrained LQG problem in~\cref{primal-LQR}.
\end{proposition}
\begin{proof}
From~\cref{thm:1}, we conclude that $\Lambda$ is the set of all KKT points. Therefore, there exists at least one $(\lambda^*,S_k^{\lambda^*}  ,F_k^{\lambda^*} ,P_k^{\lambda^*} )_{k = 0}^{N - 1} \in \Lambda$ such that the corresponding $\{(S_k^\lambda,F_k^\lambda)\}_{k=0}^{N-1}$ is an optimal solution of the constrained LQG problem in~\cref{primal-LQR}. From the statement~3) of~\cref{proposition:1}, other elements in $\Lambda$ have the same objective function value $J_p (\{ S_k^\lambda  \} _{k = 0}^{N - 1} )$. Therefore, for all $(\lambda ,S_k^\lambda  ,F_k^\lambda  ,P_k^\lambda  )_{k = 0}^{N - 1}  \in \Lambda$, the corresponding $\{(S_k^\lambda,F_k^\lambda)\}_{k=0}^{N-1}$ is an optimal solution of the constrained LQG problem in~\cref{primal-LQR}. This completes the proof.
\end{proof}

\cref{thm:2} tells us that if we can find a root $\lambda>0$ satisfying $f(\lambda)=0$, then we can find an optimal solution of~\cref{primal-LQR}. Therefore, the problem is reduced to finding a root of $f(\lambda)=0$. Our next goal is to develop a simple algorithm to solve the multi-objective LQG problem.

\subsection{Algorithm}

A natural way is to perform a line search over $\lambda \ge 0$ until $C(\{ S_k^\lambda \} _{k = 0}^{N - 1} ) = \gamma$ holds. For instance, we can gradually increase $\lambda$ from $0$ with a certain step size $\Delta \lambda >0$ until $f(\lambda)=0$ holds. Another way is to perform a bisection line search over a certain interval $\lambda  \in [0,\bar \lambda ],\bar \lambda  > 0$ and find a root $\lambda >0$ satisfying $f(\lambda ) = 0$. In this paper, we adopt the bisection search summarized in~\cref{algo:proposed}. Note that the bisection search is valid because $f$ is monotone in its argument, which is reduced to the monotonicity of $C(\{ S_k^\lambda  \} _{k = 0}^{N - 1} )$. Moreover,~\cref{algo:proposed} can be seen as a primal-dual method because they alternate the primal and dual variables updates to estimate $f$. In~\cref{algo:proposed}, $f$ can be computed using~\cref{algo:feval}.
\begin{algorithm}[h]
\caption{Primal-Dual Method with Bisection Line Search}
\begin{algorithmic}[1]
\State Input: accuracy $\varepsilon >0$, line search interval $\tilde \lambda>0$
\State Compute $f(0) = {\rm{Eval}}(0)$.
\If{$f(0)\le 0$}
\State Stop and output $0$.
\EndIf

\State Let $a = 0$ and $b = \tilde\lambda$.
\For{$k \in \{0,1,\ldots\}$}
\State Compute $f(a) = {\rm{Eval}}(a)$, $f(b) = {\rm{Eval}}(b)$, $c = (a+b)/2$, and $f(c) = {\rm{Eval}}(c)$

\If{$|(b-a)/2| \le \varepsilon$}
\State Stop and output $c$
\EndIf

\If{${\rm{sign}}(f(c)) = {\rm{sign}}(f(a))$}
\State $a \leftarrow c$
\Else{}
\State $b \leftarrow c$
\EndIf

\EndFor

\end{algorithmic}\label{algo:proposed}
\end{algorithm}

\begin{algorithm*}[h]
\caption{Policy evaluation $f(\lambda ) = {\rm{Eval}}(\lambda)$}
\begin{algorithmic}[1]
\State Input: $\lambda$

\State Dual update: Perform the recursion (Riccati equation)
\begin{align*}
A^T X_{k+1}^\lambda A-A^T X_{k+1}^\lambda B(R_k + \lambda \tilde R_k +B^T X_{k+1}^\lambda B)^{-1} B^T X_{k+1}^\lambda A+Q_k + \lambda \tilde Q_k=X_k^\lambda
\end{align*}
for all $k \in \{0,\ldots,N-1\} $ with $X_N^\lambda =Q_f + \lambda \tilde Q_f$.

\State Primal update: Compute the feedback gains
\begin{align*}
F_k=&-(R_k + \lambda \tilde R_k+B^T X_{k+1} B)^{-1} B^T X_{k+1} A,\quad k \in \{0,\ldots,N-1\}
\end{align*}

\State Primal update: Perform the recursion
\begin{align*}
&\Phi (F_k,S_{k-1})=S_k,\quad k \in \{1,\ldots,N-1\},
\end{align*}
with $\begin{bmatrix}
   I_n\\
   F_0\\
\end{bmatrix} (V+zz^T) \begin{bmatrix}
   I_n\\
   F_0\\
\end{bmatrix}^T = S_0$.

\State Compute
\begin{align*}
\tilde J_p (\{ S_k^{\lambda} \} _{k = 0}^{N - 1} ):=& {\bf{Tr}}\left( {\tilde Q_f \left( {\left[ {\begin{array}{*{20}c}
   {A^T }  \\
   {B^T }  \\
\end{array}} \right]^T S_{N - 1}^{\lambda} \left[ {\begin{array}{*{20}c}
   {A^T }  \\
   {B^T }  \\
\end{array}} \right] + W} \right)} \right) + \sum\limits_{k = 0}^{N - 1} {{\bf{Tr}}\left( {\left[ {\begin{array}{*{20}c}
   {\tilde Q_k } & 0  \\
   0 & {\tilde R_k }  \\
\end{array}} \right]S_k^{\lambda} } \right)}
\end{align*}

\State Output: $f(\lambda)=\tilde J_p (\{ S_k^{\lambda} \} _{k = 0}^{N - 1} )-\gamma$

\end{algorithmic}\label{algo:feval}
\end{algorithm*}

\subsection{Suboptimality}

Once an approximate $\lambda$ is found from~\cref{algo:proposed}, the corresponding solution $\{ S_k^\lambda,P_k^\lambda,F_k^\lambda\} _{k = 0}^{N - 1}$ can be easily found. The solution obtained by~\cref{algo:proposed} is $\varepsilon$-accurate in terms of $\lambda$, while it does not guarantee the $\varepsilon$-accuracy in terms of the objective $f(\lambda)$ or other variables $\{ S_k^\lambda,P_k^\lambda,F_k^\lambda\} _{k = 0}^{N - 1}$ induced from the $\varepsilon$-accuracy of $\lambda$, which depend on their sensitivities in $\lambda$. From the structures of $f(\lambda)$ or $\{ S_k^\lambda,P_k^\lambda,F_k^\lambda\} _{k = 0}^{N - 1}$, we can conclude that if $\lambda$ is $\varepsilon$-accurate, then $f(\lambda)$ is $\rho(\varepsilon)$-accurate, i.e. $|f(\lambda)|\le \rho(\varepsilon)$ for some function $\rho$ such that $\rho(\varepsilon) \to 0$ as $\varepsilon \to 0$. The function $\rho$ depends on the system parameters such as $(A,B)$, $Q_f,Q_k,R_k,\tilde Q_f,\tilde Q_k,\tilde R_k,k \ge 0$, and $N$.

Due to the finite precision error in the bisection search, it is hard to satisfy the equality $C(\{ S_k^\lambda \} _{k = 0}^{N - 1} ) = \gamma$ or $f(\lambda)=0$ exactly. Assume that $f(\lambda)$ is $\rho$-accurate, i.e. $|f(\lambda)|\le \rho$. This implies
\[
C(\{ S_k^\lambda  \} _{k = 0}^{N - 1} ) - \gamma = :a \in (-\rho,\rho)
\]
Then, such $\lambda >0$ is the dual variable such that
\[
C(\{ S_k^\lambda  \} _{k = 0}^{N - 1} ) = \gamma  + a
\]
is satisfied, and the corresponding tuple $\{\lambda, S_k^\lambda,P_k^\lambda,F_k^\lambda\} _{k = 0}^{N - 1}$ satisfies the KKT condition with $\gamma \leftarrow \gamma  + a$. We can conclude that $\{S_k^\lambda,F_k^\lambda\} _{k = 0}^{N - 1}$ is an optimal solution of~\cref{primal-LQR} with $\gamma$ replaced with $\gamma+ a \in (\gamma- \rho,\gamma + \rho)$.
\begin{proposition}\label{thm:3}
Suppose that given $\lambda >0$, $f(\lambda)$ is $\rho$-accurate, i.e. $|f(\lambda)|\le \rho$, and define
\[
C(\{ S_k^\lambda  \} _{k = 0}^{N - 1} ) - \gamma = :a \in (-\rho,\rho)
\]
Then, for the corresponding tuple $\{\lambda, S_k^\lambda,P_k^\lambda,F_k^\lambda\} _{k = 0}^{N - 1}$, $\{S_k^\lambda,F_k^\lambda\} _{k = 0}^{N - 1}$ is an optimal solution of~\cref{primal-LQR} with $\gamma$ replaced with $\gamma+ a \in (\gamma- \rho,\gamma + \rho)$.
\end{proposition}
\begin{proof}
The corresponding tuple $\{\lambda, S_k^\lambda,P_k^\lambda,F_k^\lambda\} _{k = 0}^{N - 1}$ satisfies the KKT condition with the complement slackness condition replaced with $\lambda (\tilde J_p (\{ S_k \} _{k = 0}^{N - 1} ) - \gamma -a ) = 0$. The proof is concluded using~\cref{thm:2}.
\end{proof}

\cref{thm:3} suggests that the solution obtained by using~\cref{algo:proposed} is a suboptimal solution of~\cref{primal-LQR} with $\gamma$ replaced with $\gamma  + a$.

\subsection{Computational efficiency}
The number of variables in the problem is upper bounded by $O(n^2 \cdot N)$. If we use an SDP to solve the multi-objective problem using interior point algorithms, the time complexity is known to be upper bounded by $O(n^6 \cdot N^3  \cdot \log (1/\varepsilon ))$ to obtain an $\varepsilon$-accurate solution~\cite{gahinet1996lmi}. Therefore, the computational time may explode cubically as $N\to \infty$. On the other hand, the bisection line search is known to find an $\varepsilon$-accurate solution within the number of iterations bounded by $O(\log (\varepsilon _0 /\varepsilon ))$, where $\varepsilon _0  = |b - a|$ is the initial bracket size. The time complexity of the proposed algorithm per iteration is $O(n^2 \cdot N)$. Therefore, the overall time complexity is bounded by $O(n^2 \cdot N \cdot \log (\varepsilon _0 /\varepsilon ))$, which is linear in $N$. Assuming that both notions of the $\varepsilon$-accuracy is reasonably compatible for fair comparisons, the proposed bisection algorithm may perform much faster than the interior-point algorithms especially when $N$ is large, which is the case in most applications. Especially, when the model predictive control is applied, where~\cref{primal-LQR} is solved at every iterations, the proposed scheme could play an important role. The compatibility of $\varepsilon$-accuracy of both approaches is hard to be addressed within the scope of this paper. We will provide numerical comparative analysis at the end of this paper to demonstrate the efficiency of the algorithm.

\subsection{Deterministic cases}

The previous results assume that the noises are stochastic and the covariance matrices $V$ and $W$ are positive definite. However, it does not cover important applications where some variables are deterministic or fixed. To cover more practical cases, we will extend the results to the generic case that some elements of the noise vectors are deterministic. In this case, $V \succ 0,W \succ 0$ should be relaxed to $V \succeq 0,W \succeq 0$. Note that if $V \succ 0,W \succ 0$, then the KKT point in~\cref{theorem1} uniquely satisfies the KKT condition for any fixed $\lambda >0$. For the deterministic case $V = W = 0$ or $V \succeq, W \succeq 0$, it still satisfies the KKT condition, while it may not be a unique solution. Therefore, we can not preserve the optimality arguments in~\cref{thm:2}. Fortunately, the previous results hold in this case under mild assumptions.
\begin{proposition}
Suppose that the second condition, $V \succ 0,W \succ 0$, in~\cref{assumption:2} is relaxed to $V \succeq 0,W \succeq 0$. Assume that $f$ is a bijection for the given $V \succeq 0,W \succeq 0$. Consider any tuples $(\lambda ,S_k^\lambda  ,F_k^\lambda  ,P_k^\lambda  )_{k = 0}^{N - 1}$ such that $\lambda>0$ and $C(\{ S_k^\lambda \} _{k = 0}^{N - 1} ) = \gamma$. Then, the corresponding $\{(S_k^\lambda,F_k^\lambda)\}_{k=0}^{N-1}$ is an optimal solution of the constrained LQG problem in~\cref{primal-LQR}.
\end{proposition}
\begin{proof}
We first prove the continuity of $(\lambda ,S_k^\lambda  ,F_k^\lambda  ,P_k^\lambda  )_{k = 0}^{N - 1}$ as a function of $V$ and $W$, and denote them by $(\lambda,S_k^{\lambda,V,W},F_k^{\lambda,V,W},P_k^{\lambda,V,W})_{k = 0}^{N - 1}$. First of all, suppose that $\lambda >0$ is fixed. Then, $F_k^{\lambda,V,W}$ and $P_k^{\lambda,V,W}$ do not depend on $V$ and $W$, and hence are continuous as functions of $(V,W)$. Moreover, $S_k^{\lambda,V,W}$ depends on $(V,W)$ linearly, and thus, is continuous in $(V,W)$. Similarly, so are $J_p^{V,W}$ and $C^{V,W}$ as functions of $(V,W)$. Now, $f^{V,W}$ is also continuous as a function of $(V,W)$ and $\lambda$. Consider the set-valued mapping $T:(V,W) \mapsto \{ \lambda  > 0:f^{V,W}(\lambda) = 0\}$. If $f$ is bijective for $V \succeq 0$ and $W \succeq 0$, then the output of $T$ is singleton, and $T(V,W)$ is the point on the graph of $f^{V,W}$ which crosses zero because the Lagrange multiplier $\lambda^* > 0$ which solves the KKT condition is a root of $f^{V,W}(\lambda)$. Therefore, from the continuity of $f^{V,W}$ on $(V,W)$ and $\lambda \geq 0$, we can prove that $T: {\mathbb R}^{n \times n} \times {\mathbb R}^{m \times m} \to {\mathbb R}_{++}$ is also continuous as follows. First of all, note that by the continuity of $f^{V,W}$ on $(V,W)$, $f^{V,W}$ is a bijection for all $(V',W')$ around $(V,W)$. In the sequel, assume that $(V',W')$ always lies inside such a set. Note also that $T(V,W) = (f^{V,W})^{-1}(0)$ and $(f^{V,W})^{-1}(y)$ is continuous in $y$ by the continuity of $f^{V,W}(\lambda)$ in $\lambda$. We will show that for any $\varepsilon >0$, there exists $\delta >0$ such that $\left\| {V - V'} \right\|_2  + \left\| {W - W'} \right\|_2  < \delta$ implies $|T(V,W) - T(V',W')| < \varepsilon$. To proceed, let us define $T(V,W) = (f^{V,W})^{-1}(0) = \lambda^*$ and $T(V',W') = (f^{V',W'})^{-1} (0) = \lambda'$. By the continuity of $f^{V,W}$ in $V$ and $W$, for any $\varepsilon' >0$, there exists $\delta >0$ such that $\left\| {V - V'} \right\|_2  + \left\| {W - W'} \right\|_2  < \delta$ implies $|f^{V,W} (\lambda ^* ) - f^{V',W'} (\lambda ^* )| < \varepsilon'$. Moreover, by the continuity of $(f^{V,W})^{ - 1}(y)$ in $y$, for any $\varepsilon >0$, there exists $\varepsilon ' >0$ such that $|x - y| < \varepsilon '$ implies $|(f^{V',W'})^{ - 1} (x) - (f^{V',W'})^{ - 1} (y)| < \varepsilon$.
With $x = f^{V,W} (\lambda ^* ),y = f^{V',W'} (\lambda ^* )$, we have
\begin{align*}
&|(f^{V',W'})^{ - 1} (f^{V,W} (\lambda ^* )) - (f^{V',W'})^{ - 1} (f^{V',W'} (\lambda ^* ))|\\
 =& |(f^{V',W'})^{ - 1} (0) - \lambda ^* |\\
 =& |(f^{V',W'})^{ - 1} (0) - (f^{V,W})^{ - 1} (0)|\\
  <& \varepsilon
\end{align*}
Therefore, this proves the continuity of $T(V,W) = (f^{V,W})^{-1}(0)$ in $V$ and $W$. Hence, $(\lambda,S_k^{\lambda,V,W}  ,F_k^{\lambda,V,W}  ,P_k^{\lambda,V,W})_{k = 0}^{N - 1}$ is continuous as a function of $V$ and $W$. Note that $\lambda$ is also a function of $(V,W)$. As the next step, consider a sequence $(V_i ,W_i )_{i = 0}^\infty$ such that $V_i  = V + (0.5)^i I$ and $W_i  = W + (0.5)^i I$ so that $(V_i,W_i) \to (V,W)$ as $i \to \infty$ and $V_i \succ 0,W_i\succ 0$ for all $i \ge 0$. Then, the corresponding $(S_k^{\lambda_i,V_i,W_i } ,F_k^{\lambda_i,V_i,W_i })_{k = 0}^{N - 1}$ is the unique optimal solution corresponding to $V_i$ and $W_i$, where $\lambda_i$ is the Lagrange multiplier corresponding to $V_i$ and $W_i$. From the continuity of the KKT point in $(V,W)$, we can arrive at the desired conclusion.
\end{proof}

\subsection{Example}
In this section, we will provide a simple example to illustrate the validity and efficiency of the proposed approach. Let us consider~\cref{ex1} again with the following setting:
\[
N = 1000,\quad \gamma  = 25000,\quad V = 0,\quad W = 0.01I,\quad z = \left[ {\begin{array}{*{20}c}
   {25}  \\
   {25}  \\
   {30}  \\
   {24}  \\
\end{array}} \right]
\]

Running the proposed algorithm with the initial search interval $[0, \tilde \lambda] = [0, 100]$ and the accuracy $\varepsilon = 0.000001$ leads to $\lambda^* = 0.2448$ with the elapsed time $1.83$ seconds. With the obtained control policy, the evolution of the indoor temperature $x_1(t)$, reference temperature $x_4(t)$, input $u(t)$, and the histograms of power of the objective function and constraint cost are depicted in~\cref{fig:2}. The histogram has been obtained over $3000$ samples, and the empirical average of the constraint cost is $25129$, which meets the inequality constraint approximately.
\begin{figure*}[t]
\centering\includegraphics[width=14cm,height=9cm]{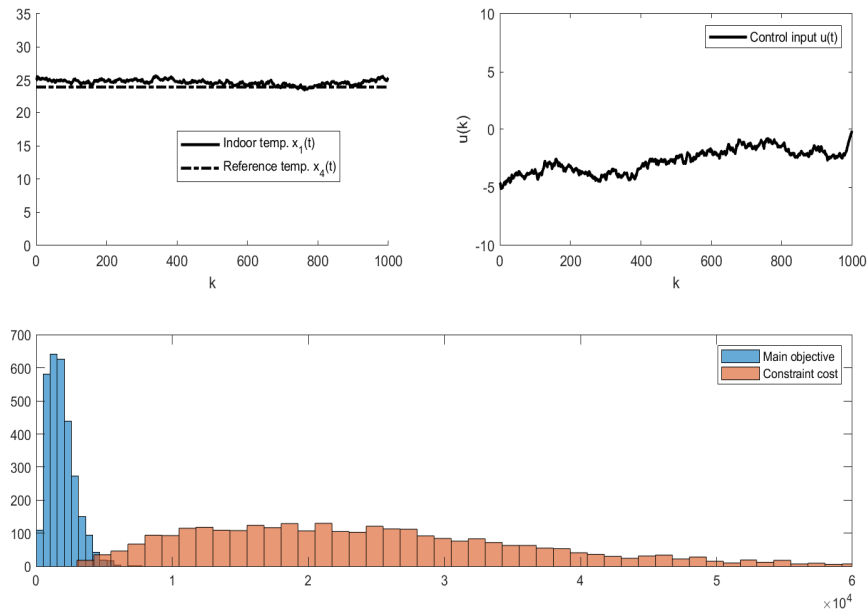}\caption{Evolution of the indoor temperature $x_1(t)$, reference temperature $x_4(t)$, and the input $u(t)$}\label{fig:2}
\end{figure*}

Running the proposed algorithm with the same setting except for $\gamma  = 10000$ leads to $\lambda^* = 0.8959$. The corresponding simulation results are given in~\cref{fig:3}. The empirical average of the constraint cost from samples in histogram is $9956$, which meets the inequality constraint approximately..
\begin{figure*}[t]
\centering\includegraphics[width=15cm,height=9cm]{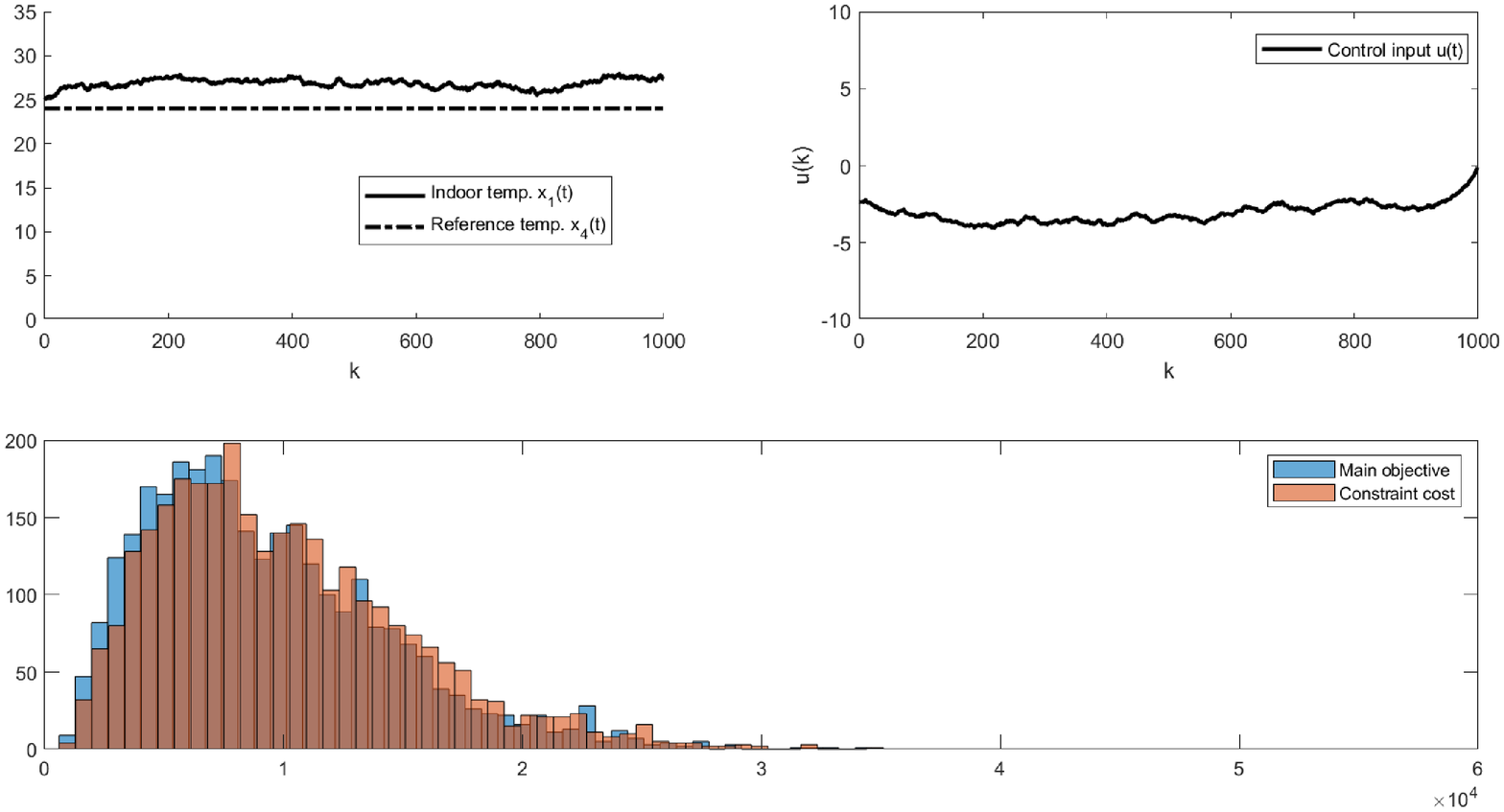}\caption{Evolution of the indoor temperature $x_1(t)$, reference temperature $x_4(t)$, and the input $u(t)$}\label{fig:3}
\end{figure*}

A semidefinite programming problem (SDP) for solving the same problem~\cref{primal-LQR} is given by
\begin{align*}
&\max _{S_k ,k \in \{ 0,1, \ldots ,N\} ,\lambda  \ge 0} \sum\limits_{k = 0}^N {{\rm{Tr}}S_k }  - \lambda \gamma\\
&{\rm{subject}}\,\,{\rm{to}}\\
&\left[ {\begin{array}{*{20}c}
   {Q_k  + \lambda \tilde Q_k } & 0  \\
   0 & {R_k  + \lambda \tilde R_k }  \\
\end{array}} \right] + \left[ {\begin{array}{*{20}c}
   {A^T S_{k + 1} A - S_k } & {A^T S_{k + 1} B}  \\
   {B^T S_{k + 1} A} & {B^T S_{k + 1} B}  \\
\end{array}} \right] \ge 0,\quad k \in \{ 0,1, \ldots ,N\}
\end{align*}
with $S_{N + 1}  = 0,Q_N  = Q_f ,\tilde Q_N  = \tilde Q_f ,R_N  = \tilde R_N  = 0$, which can be readily obtained by modifying the results in~\cite{gattami2010generalized}. The histograms of the elapsed times of the proposed algorithm and the above SDP problem to solve the building problem are shown in over $30$ samples.
\begin{figure*}[t]
\centering\includegraphics[width=15cm,height=6cm]{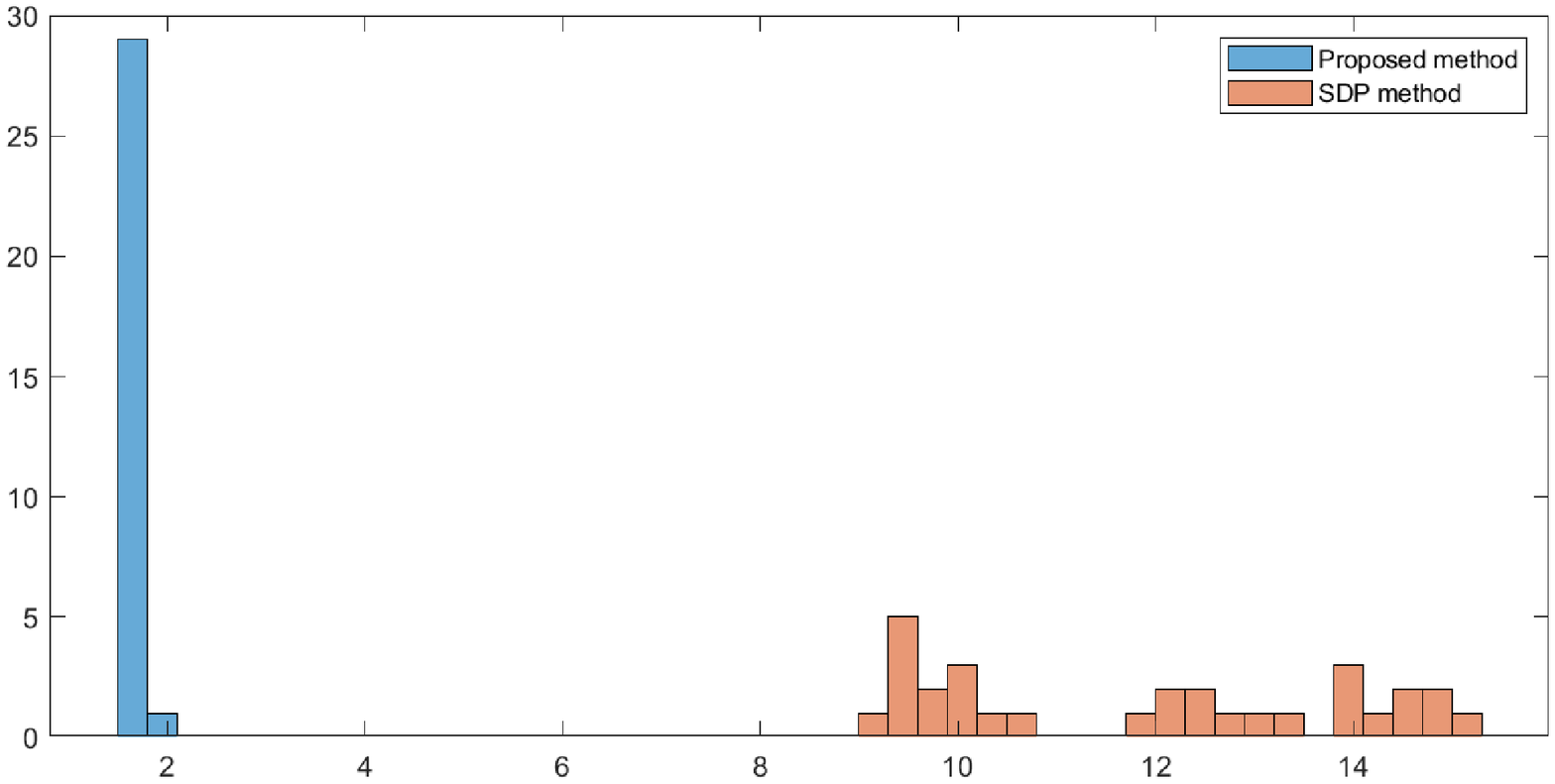}\caption{Evolution of the indoor temperature $x_1(t)$, reference temperature $x_4(t)$, and the input $u(t)$}\label{fig:4}
\end{figure*}
The average elapsed times are $1.7551$ seconds and $11.8987$ seconds for the proposed method and the SDP problem, respectively. To solve the SDP, we used SeDuMi~\cite{sturm1999using} and Yalmip~\cite{lofberg2004yalmip}. The result demonstrates that the proposed algorithm is computationally more efficient than the SDP approach.

\section{Conclusion}
In this paper, we have considered a multi-objective LQG with an input energy constraint.
An efficient bisection line search algorithm has been proposed based on optimization and Lagrangian theories.
We have rigorously analyzed optimal solutions to the underlying problem based on the KKT condition, and proved the convergence guarantees of the algorithm. It has been applied to a building control problem to demonstrate its validity and efficiency.
We expect that this work can be applied to fast model predictive control and offers new insights on LQG problems as well.
A potential future work is to consider multiple constraints for which the bisection line search becomes inefficient due to the multi-dimensional search space.

\bibliographystyle{IEEEtran}
\bibliography{reference}

\appendices


\end{document}